\documentclass{amsart}

\usepackage{amssymb,amstext,amsmath,amscd,amsthm,amsfonts,mathtools,enumerate,graphicx,latexsym}
\usepackage{hyperref}

\newtheorem{thm}{Theorem}[section]
\newtheorem{cor}[thm]{Corollary}

\newtheorem{prop}[thm]{Proposition}

\theoremstyle{definition}
\newtheorem{dfn}[thm]{Definition}

\theoremstyle{remark}
\newtheorem*{ac}{Acknowledgements}

\def\A{\operatorname{A}}
\def\B{\operatorname{B}}
\def\C{\operatorname{C}}
\def\E{\operatorname{E}}
\def\Mod{\operatorname{Mod}}
\def\depth{\operatorname{depth}}
\def\Soc{\operatorname{Soc}}
\def\Ann{\operatorname{Ann}}
\def\uleq#1{\underset{#1}{\leq}}
\def\usleq#1#2{\underset{#1\hspace{5pt}}{\leq_{#2}}}
\def\upreceq#1{\underset{#1}{\preceq}}
\def\updR{\underset{R\hspace{5pt}}{\preceq'}}
\def\updRx{\underset{R/xR\hspace{4pt}}{\preceq'}}
\def\updRy{\underset{R/yR\hspace{4pt}}{\preceq'}}
\def\updRxi{\underset{R/x_iR\hspace{5pt}}{\preceq'}}

\def\Add{\operatorname{Add}}

\def\Hom{\operatorname{Hom}}
\def\Ass{\operatorname{Ass}}

\def\m{\mathfrak{m}}
\def\p{\mathfrak{p}}
\def\q{\mathfrak{q}}

\def\P{\mathcal{P}}
\def\Q{\mathcal{Q}}
\def\X{\mathcal{X}}

\begin{document}
\allowdisplaybreaks
\title[Characterizing regular local rings via analogues of $(*)$-properties]{Characterizing regular local rings\\via analogues of $(*)$-properties}
\author{Shinnosuke Kosaka}
\address{Graduate School of Mathematics, Nagoya University, Furocho, Chikusaku, Nagoya 464-8602, Japan}
\email{kosaka.shinnosuke.d2@s.mail.nagoya-u.ac.jp}
\begin{abstract}
Let $R$ be a commutative noetherian local ring. As analogues of $(*)$-properties introduced by Ghosh, Gupta, and Puthenpurakal, we introduce and study $(\A)$-properties, $(\B)$-properties, and $(\C)$-relations. Using these three notions, we establish a criterion for a local ring to be regular. This recovers and refines the main result of Ghosh, Gupta, and Puthenpurakal and has further applications.
\end{abstract}
\maketitle

\section{Introduction}
Throughout this paper, all rings are assumed to be commutative noetherian rings. For a local ring $R$ with residue field $k$ and an integer $n\geq0$, we denote by $\Omega_R^nk$ the $n$-th syzygy of $k$ in its minimal $R$-free resolution.

Characterizing regular local rings in terms of syzygies of the residue field has been frequently studied; see \cite{D,GGP,LV,M,T} for instance. Ghosh, Gupta, and Puthenpurakal \cite[Theorem 3.1]{GGP} introduced the notion of $(*)$-properties and gave the following criterion for local rings to be regular.

\begin{thm}[Ghosh--Gupta--Puthenpurakal]\label{thmGGP}
Let $R$ be a local ring with residue field $k$ and $\P$ a $(*)$-property. Suppose that there exists a surjective $R$-module homomorphism $f:\bigoplus_{n\in\Lambda}(\Omega_R^nk)^{\oplus j_n}\to L$ where $\Lambda$ is a finite set of non-negative integers, $j_n$ is a positive integer, and $L\neq0$ satisfies $\P$. Then $R$ is regular.
\end{thm}

Here, a property $\P$ of finitely generated modules over local rings is called a {\em $(*)$-property} if it satisfies the following two conditions for all local rings $R$ and all finitely generated $R$-modules $M$:

\begin{enumerate}
\item[$(*1)$]
If $M$ satisfies $\P$ as an $R$-module, then $M/xM$ satisfies $\P$ as an $R/xR$-module for every $R$-regular element $x$.
\item [$(*2)$]
If $\depth R=0$ and $M$ satisfies $\P$ as an $R$-module, then $\Ann_RM=0$.
\end{enumerate}

Typical examples of a $(*)$-property are being a semidualizing module and being a nonzero maximal Cohen--Macaulay module of finite injective dimension. In what follows, removing the assumption of finite generation, we say that a property $\P$ of modules over local rings a {\em $(*)$-property} if $\P$ satisfies the above two conditions $(*1)$ and $(*2)$ for all local rings $R$ and all $R$-modules $M$.

We say that $\leq$ is a {\em relation of modules over local rings} (or simply a {\em relation}) when for every local ring $R$, it refers to a relation $\uleq{R}$ over $R$-modules. Theorem \ref{thmGGP} is generalized as follows:

\begin{thm}[Corollary \ref{cor property relation}]\label{thm intro}
Let $\P$, $\Q$ be properties of (possibly infinitely generated) modules over local rings and $\leq$ a relation of modules over local rings. Assume that $\P$, $\Q$, and $\leq$ satisfy the following conditions for all local rings $(R,\m,k)$ and all $R$-modules $X$, $Y$:
\begin{enumerate}[\rm(1)]
\item
If $\depth R>0$ and $X$ satisfies $\P$ (resp. $\Q$) as an $R$-module, then there exist finitely many non-maximal prime ideals $\p_1,\dots,\p_n$ such that $X/xX$ satisfies $\P$ (resp. $\Q$) as an $R/xR$-module for every $R$-regular element $x\in\m\setminus(\m^2\cup(\bigcup_{i=1}^n\p_i))$.
\item
If $\depth R=0$ and $X$ satisfies $\P$ as an $R$-module, then $\Soc R\not\subset\Ann_RX$.
\item
If $\depth R=0$, $X$ satisfies $\Q$ as an $R$-module, and $R\neq k$, then $\Soc R\subset\Ann_RX$.
\item
If $\depth R>0$ and $X\uleq{R}Y$, then there exist finitely many non-maximal prime ideals $\p_1,\dots,\p_n$ such that $X/xX\uleq{R/xR}Y/xY$ for every $R$-regular element $x\in\m\setminus(\m^2\cup(\bigcup_{i=1}^n\p_i))$.
\item
If $\depth R=0$ and $X\uleq{R}Y$, then $\Ann_RX\subset\Ann_RY$.
\end{enumerate}
A local ring $R$ is regular if there exist $R$-modules $M,N$ such that $M$ satisfies $\Q$, $N$ satisfies $\P$, and $M\uleq{R} N$.
\end{thm}

Let $R$ be a local ring. We denote by $\Mod R$ the category of $R$-modules. For a full subcategory $\X$ of $\Mod R$, we denote by $\Add\X$ the full subcategory of $\Mod R$ whose objects are all direct summands of direct sums of objects in $\X$. An $R$-module $L$ is called a {\em balanced big Cohen--Macaulay module} if every system of parameters of $R$ is an $L$-sequence. A finitely generated $R$-module $L$ is called {\em deep} if $\depth L\geq\depth R$.

As an application of Theorem \ref{thm intro}, we obtain the following criterion for the regularity.

\begin{cor}[Proposition \ref{ex C relation}]\label{cor intro}
Let $R$ be a local ring with residue field $k$. Let $M$ be an $R$-module in $\Add\{\Omega_R^nk\mid n\geq0\}$ and $N$ an $R$-module satisfying some $(*)$-property. Then $R$ is regular if either of the following conditions holds:
\begin{enumerate}[\rm(1)]
\item 
There exists an exact sequence $M\to N\to0$ of $R$-modules.
\item 
There exists an exact sequence $0\to N\to M\to L\to 0$ of $R$-modules such that $L$ is a balanced big Cohen--Macaulay $R$-module, or belongs to $\Add K$ for some deep $R$-module $K$.
\end{enumerate}
\end{cor}

It is natural to ask whether $R$ is regular if there is an injective $R$-module homomorphism from a special module to a finite direct sum of syzygies of $k$. According to \cite[Example 3.9]{GGP}, it is not true in general. However, Corollary \ref{cor intro} says that this statement is true if we impose a suitable assumption on the cokernel of the homomorphism. Moreover, Corollary \ref{cor intro} yields the following simple criterion for the regularity, which extends a theorem of Dutta \cite[Corollary 1.3]{D} to modules that are not necessarily finitely generated.

\begin{cor}[Proposition \ref{ex A property}]\label{cor intro faithfully flat}
Let $R$ be a local ring with residue field $k$. Suppose that there exists a surjective homomorphism from a direct sum of syzygies of $k$ to a faithfully flat R-module. Then R is regular. 
\end{cor}

In the next section we shall prove Theorem \ref{thm intro} and Corollaries \ref{cor intro}, \ref{cor intro faithfully flat}.

\section{Main Results}

Let $\P$ be a property of modules over local rings. For a local ring $R$, we denote by $\mathcal{P}(R)$ the full subcategory of $\Mod R$ consisting of $R$-modules that satisfy $\P$. We begin with introducing the notions of $(\A)$- and $(\B)$-properties and their variants.

\begin{dfn}
Consider the following conditions for a property $\P$ of modules over local rings.
\begin{enumerate}
\item[(P1)]
Let $R$ be a local ring with maximal ideal $\m$. If $M\in \P(R)$, then $M/xM \in \P(R/xR)$ for every $R$-regular element $x\in\m\setminus\m^2$.
\item[(P1')] 
Let $R$ be a local ring with maximal ideal $\m$. If $\depth R>0$ and $M\in \P(R)$, then there exist finitely many non-maximal prime ideals $\p_1,\dots,\p_n$ such that $M/xM \in \P(R/xR)$ for every $R$-regular element $x\in\m\setminus(\m^2\cup(\bigcup_{i=1}^n\p_i))$.
\item[(P2)]
Let $R$ be a local ring. If $\depth R=0$ and $M\in\P(R)$, then $\Soc R\not\subset\Ann_RM$.
\item[(P3)] 
Let $R$ be a local ring with residue field $k$. If $\depth R=0$, $M\in\P(R)$, and $R\neq k$, then $\Soc R\subset\Ann_RM$.
\end{enumerate}
We say that $\P$ is an {\em $(\A)$-property} (resp. {\em $(\A')$-property}) if $\P$ satisfies (P1) and (P2) (resp. (P1') and (P2)). We say that $\P$ is a {\em $(\B)$-property} (resp. {\em $(\B')$-property}) if $\P$ satisfies (P1) and (P3) (resp. (P1') and (P3)).
\end{dfn}

It is clear that (P1) implies (P1'). Therefore, an $(\A)$-property (resp. $(\B)$-property) is an $(\A')$-property (resp. $(\B')$-property). We should notice that a $(*)$-property is an $(\A)$-property; in fact, $(*1)$ (resp. $(*2)$) implies (P1) (resp. (P2)).

We give some examples of $(\A)$-properties.

\begin{prop}\label{ex A property}
We define the properties $\P_1$, $\P_2$, $\P_3$ as follows:
\begin{enumerate}[\rm(1)]
\item
We say that a module $M$ over a local ring $R$ satisfies $\P_1$ if $M$ is a semidualizing $R$-module.
\item 
We say that a module $M$ over a local ring $R$ satisfies $\P_2$ if $M$ is a non-zero maximal Cohen--Macaulay $R$-module of finite injective dimension.
\item
We say that a module $M$ over a local ring $R$ satisfies $\P_3$ if $M$ is a faithfully flat $R$-module.
\end{enumerate}
Then, $\P_1$, $\P_2$, $\P_3$ are all $(\A)$-properties.
\end{prop}

\begin{proof}
Since $\P_1$ and $\P_2$ have already been shown to be $(*)$-properties in \cite[Example 2.4 and 2.5]{GGP}, we only prove that $\P_3$ is an $(\A)$-property. Let $(R,\m)$ be a local ring. If $M$ is a faithfully flat $R$-module and $x\in\m\setminus\m^2$, then $M/xM$ is a faithfully flat $R/xR$-module by base change. Next, assume that $M$ is a faithfully flat $R$-module and $\depth R=0$. Then $\Soc R\ne0$, and since $M$ is faithfully flat, we have $(\Soc R)M\simeq(\Soc R)\otimes M\ne0$. Hence $\Soc R\not\subset\Ann_RM$.
\end{proof}

We also give some examples of $(\B)$-properties.

\begin{prop}\label{ex B property}
We define the properties $\P_4$, $\P_5$ as follows:
\begin{enumerate}[\rm(1)]
\item
We say that a module $M$ over a local ring $R$ satisfies $\P_4$ if $M\in\Add\{\Omega_R^nk\mid n\geq0\}$, where $k$ denotes the residue field of $R$.
\item 
We say that a module $M$ over a local ring $R$ satisfies $\P_5$ if $R$ is a Cohen--Macaulay local ring with a canonical module $\omega$ and $M\in\Add\{(\Omega_R^nk)^\dag\mid n\geq\dim R\}$, where $k$ is the residue field of $R$ and $(-)^\dag\coloneqq\Hom_R(-,\omega)$.
\end{enumerate}
Then both $\P_4$ and $\P_5$ are $(\B)$-properties.
\end{prop}

\begin{proof}
We prove that $\P_5$ is a $(\B)$-property. The other can be proved similarly.

Let $(R,\m,k)$ be a local ring, and put $d=\dim R$. Assume that an $R$-module $M$ is a direct summand of $\bigoplus_{n\geq d}((\Omega_R^nk)^\dag)^{\oplus j_n}$ and $x\in \m\setminus\m^2$ is an $R$-regular element; note then that $d>0$. We set $\overline{(-)}\coloneqq(-)\otimes_RR/xR$. Then $\dim\overline{R}=d-1$. For any $n\geq d$, $\Omega_R^nk$ is a maximal Cohen--Macaulay module by the depth lemma, and therefore $\overline{(\Omega_R^nk)^\dag}\simeq(\overline{\Omega_{R}^nk})^\dag$ by \cite[Proposition 3.3.3]{BH}. Moreover, $\overline{\Omega_{R}^nk}\simeq\Omega_{\overline{R}}^nk\oplus\Omega_{\overline{R}}^{n-1}k$ by \cite[Corollary 5.3]{T}. Hence, $\overline{M}$ is a direct summand of
\[\overline{\bigoplus_{n\geq d}((\Omega_R^nk)^\dag)^{\oplus j_n}}\simeq \bigoplus_{n\geq d}(((\Omega_{\overline{R}}^nk)^\dag)^{\oplus j_n}\oplus((\Omega_{\overline{R}}^{n-1}k)^\dag)^{\oplus j_n}).\]
Therefore, $\P_5$ satisfies (P1).

Next, assume $d=0$ and $M$ is a direct summand of $\bigoplus_{n\geq 0}((\Omega_R^nk)^\dag)^{\oplus j_n}$. We have $\Soc R\subset\Ann_R(\Omega_R^nk)$ for all $n\geq0$ by \cite[Lemma 2.1]{GGP}, and therefore
\[\Soc R\subset \bigcap_{n\geq0}\Ann_R(\Omega_R^nk)=\bigcap_{n\geq0}\Ann_R(\Omega_R^nk)^\dag\subset\Ann_R\Bigl(\bigoplus_{n\geq 0}((\Omega_R^nk)^\dag)^{\oplus j_n}\Bigr)\subset\Ann_RM.\]
Hence, $\P_5$ satisfies (P2).
\end{proof}

Next, we consider modules satisfying some $(\A)$- or $(\A')$- or $(\B)$- or $(\B')$-property.

\begin{dfn}\label{dfn PE}
Let $\E\in\{\A,\A',\B,\B'\}$. We say that a module $M$ satisfies $\P_{\E}$ if there exists an $(\E)$-property which $M$ satisfies.
\end{dfn}

It is easy to see that $\P_{\A}$, $\P_{\A'}$, $\P_{\B}$, $\P_{\B'}$ are respectively $(\A)$-, $(\A')$-, $(\B)$-, $(\B')$-properties. Moreover, one has $\P_{\A}(R)\subset\P_{\A'}(R)$ and $\P_{\B}(R)\subset\P_{\B'}(R)$ for every local ring $R$ since $(\A)$- and $(\B)$-properties are respectively $(\A')$- and $(\B')$-properties.

With these preparations, we obtain the following criterion for a local ring to be regular.

\begin{thm}\label{thm regularity}
Let $R$ be a local ring. If there exists an $R$-module $M$ such that $M\in \P_{\A'}(R)\cap\P_{\B'}(R)$, then $R$ is regular.
\end{thm}

\begin{proof}
We use induction on $t\coloneqq\depth R$. If $t=0$, then $R=k$ by (P2) and (P3). Hence $R$ is regular.

Suppose that $t>0$. By (P2) and (P3), there exist non-maximal prime ideals $\p_1,\dots,\p_m$ and $\q_1,\dots\q_n$ such that $M/xM \in \P_{\A'}(R/xR)$ for every $R$-regular element $x\in\m\setminus(\m^2\cup(\bigcup_{i=1}^m\p_i))$ and $M/yM \in \P_{\B'}(R/yR)$ for every $R$-regular element $y\in\m\setminus(\m^2\cup(\bigcup_{j=1}^n\q_j))$. Now we can take an $R$-regular element $x\in\m\setminus(\m^2\cup(\bigcup_{i=1}^m\p_i)\cup(\bigcup_{j=1}^n\q_j))$ by prime avoidance. Then we have $M/xM \in \P_{\A'}(R/xR)\cap\P_{\B'}(R/xR)$ and therefore, by the induction hypothesis, $R/xR$ is regular. Since $x\in\m\setminus\m^2$, $R$ is regular.
\end{proof}

We will observe later that Theorem \ref{thm regularity} gives rise to several applications. We introduce $(\C)$- and $(\C')$-relations.

\begin{dfn}
Let $\leq$ be a relation of modules over local rings. Consider the following conditions for $\leq$.
\begin{enumerate}
\item[(R1)]
Let $R$ be a local ring with maximal ideal $\m$, and let $M,N$ be $R$-modules. If $M\uleq{R}N$, then $M/xM\uleq{R/xR}N/xN$ for every $R$-regular element $x\in\m\setminus\m^2$.
\item[(R1')] 
Let $R$ be a local ring with maximal ideal $\m$, and let $M,N$ be $R$-modules. If $\depth R>0$ and $M\uleq{R}N$, then there exist finitely many non-maximal prime ideals $\p_1,\dots,\p_n$ such that $M/xM\uleq{R/xR}N/xN$ for every $R$-regular element $x\in\m\setminus(\m^2\cup(\bigcup_{i=1}^n\p_i))$.
\item[(R2)]
Let $R$ be a local ring, and let $M,N$ be $R$-modules. If $\depth R=0$ and $M\uleq{R}N$, then $\Ann_RM\subset\Ann_RN$.
\end{enumerate}
We say that $\leq$ is a {\em $(\C)$-relation} (resp. {\em $(\C')$-relation}) if $\leq$ satisfies (R1) and (R2) (resp. (R1') and (R2)).
\end{dfn}

It is clear that (R1) implies (R1'). Therefore, a $(\C)$-relation is a $(\C')$-relation.

We give some examples of $(\C)$- and $(\C')$-relations.

\begin{prop}\label{ex C relation}
We define the relations $\leq_1,\leq_2,\leq_3,\leq_4$ as follows:
\begin{enumerate}[\rm(1)]
\item 
We say that modules $M,N$ over a local ring $R$ satisfy $M\usleq{R}{1} N$ if there exists a surjective $R$-module homomorphism $M\to N$.
\item 
We say that modules $M,N$ over a local ring $R$ satisfy $M\usleq{R}{2} N$ if there exists a short exact sequence
\[0\to N\to M\to L\to 0\]
of $R$-modules such that $L$ is a balanced big Cohen--Macaulay $R$-module.
\item
We say that modules $M,N$ over a local ring $R$ satisfy $M\usleq{R}{3} N$ if there exists a short exact sequence
\[0\to N\to M\to L\to 0\]
of $R$-modules such that $L\in\Add K$ for some deep $R$-module $K$.
\item 
We say that modules $M,N$ over a local ring $R$ satisfy $M\usleq{R}{4} N$ if there exists an $R$-module $L$ such that $N\simeq M\otimes_RL$.
\end {enumerate}
Then, $\leq_1$, $\leq_2$, $\leq_4$ are all $(\C)$-relations, while $\leq_3$ is a $(\C')$-relation.
\end{prop}

\begin{proof}
Let $R$ be a local ring with maximal ideal $\m$.

$\leq_1$: Let $M,N$ be $R$-modules and assume that there is a surjective  $R$-module homomorphism $M\to N$. Then the induced map $M/xM\to N/xN$ is a surjective $R$-module homomorphism for every $x\in R$, and $\Ann_RM\subset\Ann_RN$.

$\leq_2$: This is proved in the same way as $\leq_3$ below.

$\leq_3$: Let $M,N$ be $R$-modules and assume that there exists a short exact sequence
\[0\to N\to M\to L\to 0\]
of $R$-modules such that $L\in\Add K$ for some deep $R$-module.

Suppose $\depth R>0$. Since $\depth K>0$, the set $\{\p_1,\dots,\p_n\}\coloneqq\Ass K$ consists of non-maximal prime ideals of $R$. Let $x\in\m\setminus(\m^2\cup(\bigcup_{i=1}^n\p_i))$ be an $R$-regular element. Since $x$ is an $L$-regular element, we have a short exact sequence
\[0\to N/xN\to M/xM\to L/xL\to0\]
of $R/xR$-modules by \cite[Proposition 1.1.4]{BH}. Moreover, $K/xK$ is a deep $R/xR$-module and $L/xL\in \Add K$. Therefore, $\leq_3$ satisfies (R1').

In addition, $\Ann_RM\subset\Ann_RN$ since the map $N\to M$ is an injective $R$-module homomorphism. Hence, $\leq_3$ satisfies (R2).

$\leq_4$: Let $M,N$ be $R$-modules and assume that there exists an $R$-module $L$ such that $N\simeq M\otimes_RL$. Then, we have $N/xN\simeq M/xM\otimes_{R/xR}L/xL$ for every $x\in R$, and $\Ann_RM\subset\Ann_RN$.
\end{proof}

In Definition \ref{dfn PE}, we defined the property $\P_{\E}$ for each $\E\in\{\A,\A',\B,\B'\}$. Similarly to this, we make the following definition.

\begin{dfn}\label{dfn preceq}
We say that modules $M,N$ over a local ring $R$ satisfy $M\upreceq{R}N$ (resp. $M\updR N$) if there exists a $(\C)$-relation (resp. $(\C')$-relation) $\leq$ such that $M\uleq{R}N$.
\end{dfn}

It is easy to see that $\preceq$, $\preceq'$ are respectively $(\C)$-, $(\C')$-relations. Moreover, for modules $M,N$ over a local ring $R$, if $M\upreceq{R}N$, then $M\updR N$, since a $(\C)$-relation is a $(\C')$-relation.

We present several properties of $\preceq$ and $\preceq'$.

\begin{prop}\label{prop property relation}
Let $(R,\m)$ be a local ring, and let $X,Y$ be $R$-modules. Then the following hold:
\begin{enumerate}[\rm(1)]
\item 
If $X\upreceq{R}Y$ and $Y\in\P_{\A}(R)$, then $X\in\P_{\A}(R)$.
\item 
If $X\updR Y$ and $Y\in\P_{\A'}(R)$, then $X\in\P_{\A'}(R)$.
\item 
If $X\upreceq{R}Y$ and $X\in\P_{\B}(R)$, then $Y\in\P_{\B}(R)$.
\item 
If $X\updR Y$ and $X\in\P_{\B'}(R)$, then $Y\in\P_{\B'}(R)$.
\end{enumerate}
\end{prop}

\begin{proof}
We prove only (2). The others can be proved in the same way. We define the property $\P$ as follows:
\begin{center}
We say that a module $M$ over a local ring $R$ satisfies $\P$ if there exists\\an $R$-module $N\in\P_{\A'}(R)$ such that $M\updR N$.
\end{center}
Since $X\in \P(R)$, it suffices to prove that $\P$ is an $(\A')$-property.

(P1') Suppose that $M\in \P(R)$ and $\depth R>0$. Then there exists $N\in\P_{\A'}(R)$ such that $M\updR N$. Hence, there exist non-maximal prime ideals $\p_1,\dots,\p_m$ and $\q_1,\dots\q_n$ such that $N/xN \in \P_{\A'}(R/xR)$ for every $R$-regular element $x\in\m\setminus(\m^2\cup(\bigcup_{i=1}^m\p_i))$ and $M/yM\updRy N/yN$ for every $R$-regular element $y\in\m\setminus(\m^2\cup(\bigcup_{j=1}^n\q_j))$. Then for every $R$-regular element $x\in\m\setminus(\m^2\cup(\bigcup_{i=1}^m\p_i)\cup(\bigcup_{j=1}^n\q_j))$, one has $N/xN \in \P_{\A'}(R/xR)$ and $M/xM\updRx N/xN$, and therefore $M/xM\in\P(R/xR)$.

(P2) Assume that $\depth R=0$. If $M\in\P(R)$, then there exists $N\in\P_{\A'}(R)$ such that $M\updR N$. Since $\Soc R\not\subset\Ann_RN$ and $\Ann_RM\subset\Ann_RN$, we have $\Soc R\not\subset\Ann_RM$.

Therefore, $\P$ is an $(\A')$-property and this completes the proof of the proposition.
\end{proof}

By this proposition, we obtain a generalization of Theorem \ref{thmGGP}.

\begin{cor}\label{cor property relation}
Let $R$ be a local ring. Suppose that there exist $R$-modules $M\in\P_{\B'}(R)$ and $N\in\P_{\A'}(R)$ such that $M\updR N$. Then R is regular.
\end{cor}

\begin{proof}
By Proposition \ref{prop property relation}, we have $M\in\P_{\A'}(R)$. Since $M\in\P_{\A'}(R)\cap\P_{\B'}(R)$, $R$ is regular by Theorem \ref{thm regularity}.
\end{proof}

Corollary \ref{cor intro} follows by combining the above corollary with Propositions \ref{ex B property} and \ref{ex C relation}.

The proposition below shows the transitivity of the relations $\preceq$ and $\preceq'$.

\begin{prop}\label{prop transitivity}
Let $(R,\m)$ be a local ring, and let $X,Y,Z$ be $R$-modules. Then the following hold:
\begin{enumerate}[\rm(1)]
\item 
If $X\upreceq{R}Y$ and $Y\upreceq{R}Z$, then $X\upreceq{R}Z$.
\item 
If $X\updR Y$ and $Y\updR Z$, then $X\updR Z$.
\end{enumerate}
\end{prop}

\begin{proof}
We prove only (2), since (1) can be proved in the same way. We define the relation $\leq$ as follows:
\begin{center}
We say that modules $M,N$ over a local ring $R$ satisfy $M\uleq{R}N$ if there exists\\an $R$-module $L$ such that $M\updR L$ and $L\updR N$.
\end{center}
Since $M\uleq{R}N$, it suffices to prove that $\leq$ is a $(\C')$-relation.

(R1') Suppose that $R$-modules $M,N$ satisfy $M\uleq{R}N$ and $\depth R>0$. Then there exists an $R$-module $L$ such that $M\updR L$ and $L\updR N$. Hence, there exist non-maximal prime ideals $\p_1,\dots,\p_m$ and $\q_1,\dots\q_n$ such that $M/xM\updRx L/xL$ for every $R$-regular element $x\in\m\setminus(\m^2\cup(\bigcup_{i=1}^m\p_i))$ and $L/yL\updRy N/yN$ for every $R$-regular element $y\in\m\setminus(\m^2\cup(\bigcup_{j=1}^n\q_j))$. Then for every $R$-regular element $x\in\m\setminus(\m^2\cup(\bigcup_{i=1}^m\p_i)\cup(\bigcup_{j=1}^n\q_j))$, one has $M/xM\updRx L/xL$ and $L/xL\updRx N/xN$, and therefore $M/xM\uleq{R/xR}N/xN$.

(R2) Suppose that $R$-modules $M,N$ satisfy $M\uleq{R}N$ and $\depth R=0$. Then there exists an $R$-module $L$ such that $M\updR L$ and $L\updR N$. Since $\Ann_RM\subset\Ann_RL$ and $\Ann_RL\subset\Ann_RN$, we have $\Ann_RM\subset\Ann_RN$.

Therefore, $\leq$ is a $(\C')$-relation and this completes the proof of the proposition.
\end{proof}

We can extend Corollary \ref{cor property relation} as follows by using the above proposition.

\begin{cor}\label{cor transitivity}
Let $R$ be a local ring. Suppose that there exist a finite number of $R$-modules $M_1,M_2,\dots,M_n$ such that $M_1\in\P_{\B'}(R)$ and $M_n\in\P_{\A'}(R)$, and $M_1\updR M_2\updR\cdots\updR M_n$. Then R is regular.
\end{cor}

The following proposition shows that $\preceq$ and $\preceq'$ behave well under direct sums.

\begin{prop}\label{prop add}
Let $(R,\m)$ be a local ring. Let $X$ be an $R$-module and $\{X_i\}_{i\in I}$ a family of $R$-modules. Then the following hold:
\begin{enumerate}[\rm(1)]
\item 
If $X\upreceq{R}X_i$ for every $i\in I$ and $Y\in\Add\{X_i\mid i\in I\}$, then $X\upreceq{R}Y$.
\item 
If $I$ is a finite set and $X\updR X_i$ for every $i\in I$, and $Y\in\Add\{X_i\mid i\in I\}$, then $X\updR Y$.
\end{enumerate}
\end{prop}

\begin{proof}
We prove only (2), since (1) can be proved in the same way. We define the relation $\leq$ as follows:
\begin{center}
We say that modules $M,N$ over a local ring $R$ satisfy $M\uleq{R}N$ if there exists a family of\\$R$-modules $\{M_i\}_{i\in I}$ such that $M\updR M_i$ for every $i\in I$ and $N\in\Add\{M_i\mid i\in I\}$.
\end{center}
Since $M\uleq{R}N$, it suffices to prove that $\leq$ is a $(\C')$-relation.

(R1') Suppose that $R$-modules $M,N$ satisfy $M\uleq{R}N$ and $\depth R>0$. Then there exists a family of $R$-modules $\{M_i\}_{i\in I}$ such that $M\updR M_i$ for every $i\in I$ and $N\in\Add\{M_i\mid i\in I\}$. Hence, there exist non-maximal prime ideals $\p_{i,j}$ ($i\in I$, $1\leq j\leq n_i$) such that $M/x_iM\updRxi N/x_iN$ for every $i\in I$ and $R$-regular element $x_i\in\m\setminus(\m^2\cup(\bigcup_{j=1}^{n_i}\p_{i,j}))$. Then for every $R$-regular element $x\in\m\setminus(\m^2\cup(\bigcup_{i\in I}\bigcup_{j=1}^{n_i}\p_{i,j}))$, one has $M/xM \updRx M_i/xM_i$ for every $i\in I$ and $N/xN\in\Add\{M_i/xM_i\mid i\in I\}$, and therefore $M/xM\uleq{R}N/xN$.

(R2) Suppose that $R$-modules $M,N$ satisfy $M\uleq{R}N$ and $\depth R=0$. Then there exists a family of $R$-modules $\{M_i\}_{i\in I}$ such that $M\updR M_i$ for every $i\in I$ and $N\in\Add\{M_i\mid i\in I\}$. Since $\Ann_RM\subset\Ann_RM_i$ for every $i\in I$, we have
\[\Ann_RM\subset\bigcap_{i\in I}\Ann_RM_i=\Ann_R\Bigl(\bigoplus_{i\in I}M_i\Bigr)\subset\Ann_RN.\]

Therefore, $\leq$ is a $(\C')$-relation and this completes the proof of the proposition.
\end{proof}

We close the section by presenting another extension of Corollary \ref{cor property relation}, which follows from the above proposition.

\begin{cor}
Let $R$ be a local ring. Let $M,N$ be $R$-modules and $\{M_i\}_{i\in I}$ a family of $R$-modules such that $M\in \P_{B'}(R)$ and $N\in \P_{A'}(R)\cap\Add\{M_i\mid i\in I\}$. Then $R$ is regular if either of the following conditions holds:
\begin{enumerate}[\rm(1)]
\item 
$M\upreceq{R}M_i$ for every $i\in I$.
\item 
$I$ is a finite set and $M\updR M_i$ for every $i\in I$.
\end{enumerate}
\end{cor}

\begin{ac}
The author would like to thank his supervisor Ryo Takahashi for giving many thoughtful questions and helpful discussions. He also thanks Kaito Kimura and Yuki Mifune, and Yuya Otake for their valuable comments.
\end{ac}

\end{document}